\documentclass[a4paper,12pt]{amsart}
\usepackage{amsmath,amssymb, amsbsy,amsmath,amsthm}

\newtheorem{theorem}{Theorem}

\newtheorem{lemma}[theorem]{Lemma}

\newtheorem{remark}[theorem]{\it Remark}

\def\RR{{\mathcal R}}
\def\HH{{\mathcal H}}
\def\KK{{\mathcal K}}

\begin{document}

\title
[resonances]{Exponential decay and resonances in a driven system}

\author{Philippe Briet*}
\address{*Universite de Toulon and CPT-CNRS.}
 \email{briet@univ-tln.fr}

\author{Claudio Fern\'andez**}
\address{**Pontificia Universidad Cat\'olica de Chile. Facultad de Matem\'aticas}
\email{cfernand@mat.puc.cl}


%\dedicatory{}


\begin{abstract}
We study the resonance phenomena for time periodic perturbations
of a Hamiltonian $H$ on the Hilbert space $L^2( \mathbb R ^d)$. Here, resonances are
characterized in terms of  time  behavior of  the survival probability. Our approach uses the Floquet-Howland
formalism combined with the results  of L. Cattaneo, J.M. Graf  and W. Hunziker on resonances  for time independent
perturbations.

\end{abstract}

\date{}

\maketitle

\section{Introduction}
Let $\{H(t)\}$ be a  periodic time dependent quantum Hamiltonian, that is, a
family of self--adjoint operators acting on a complex Hilbert space $\mathcal H$. These
operators describe   a quantum driven system whose states are given
by the solution of the Schr\"odinger equation:
\begin{equation}
(-{\hbar\over i}\partial_t+H(t))\psi(t)=0,\quad\psi(t=0)=\psi_0.
\end{equation}
The main question we want to address concerns the possibility of existence of
metastable  states  (resonances)  for such a system, in the case where $H(t)$ is a perturbation
of a free time independent Hamiltonian $H_0$ having  bound states.   

Here, we characterize the presence of a resonance
in a dynamical fashion, in terms of an approximate exponential behavior of the
associated time evolution.

We assume that $\hbar =1$ and consider the propagator $U(t,s); t,s >0$,
associated to the Schr\"odinger equation:
\begin{equation}
i{\partial \psi\over \partial t} =H(t)\psi(t).
\end{equation}
When $H_0 \varphi=E_0\varphi$, we expect that $\varphi$ becomes a resonant
state for the Hamiltonian $H(t)$, in the sense that,
$$
\langle \varphi, U(t,0) \varphi \rangle \approx e^{-i\lambda |t|}\, ,
$$
for  some $\lambda \in \mathbb C, \Im \lambda <0$, close to $E_0$.

The evolution for time-dependent Hamiltonians has been
considered by many authors. For example, spectral and scattering theory for
this problem has been treated by J. Howland in several articles (see
e.g.  \cite{howland1} and \cite{howland2}).

More recently,  in \cite{yajima1} and \cite{yajima2}, the authors
have considered  a perturbation of the free Laplacian by a time-periodic potential and prove the
absolute continuity of the Floquet spectrum.

There are also  some  results on the characterization of the resonance
phenomenon for a time independent Hamiltonian $H$, in terms of  local  exponential decay in time
of the  evolution  $\langle \varphi,  e^{-iHt}\varphi  \rangle$  for  an adequate resonance state $\varphi$  see e.g. \cite {CoSo, hu1, JeNe, MeSi, pf1, sk2, sowe}.

The relation between resonances and time decay of the evolution
can be seen from the following formula, which expresses the evolution group
as the Fourier transform of the derivative of the spectral measure,
\begin{equation}
 \label{T1}
\langle \varphi, e^{-iHt} \varphi\rangle =
\frac{1}{2\pi
i}\int_0^{\infty}e^{-itE}\langle \varphi, Im \;
(H-E-i0)^{-1}\varphi \rangle dE
\end{equation}
In several cases, the function $
F(E)\equiv \langle \varphi,\; (H-E-i0)^{-1}\varphi \rangle
$
can be continued in the lower half plane, through the cut due  to the presence of the continuous spectrum. If this function has  a pole at the complex number
$E_0-i\Gamma, \Gamma >0$, then, by deforming the contour of integration and
using residue calculus,  it  can be proven that
$$
 \langle \varphi, e^{-iHt} \varphi\rangle = e^{-itE_0-t\Gamma}\|\varphi\|^2,
$$
which is slowly (and exponentially) decaying, if $\Gamma$ is
small.

Mathematical justification of this result (the single-pole
approximation) is quite difficult and requires strong conditions. We mention for instance \cite{abbfr, ge2, hu1, sk1, sk2}.
For a concrete one dimensional model a  different approach appears in \cite{lav2}.

 A  more recent   result on the   dynamical characterization of the resonance states  was proposed
 in  \cite{CGHu}. It is based on  the  positive commutator theory of E. Mourre \cite{Mo} and in
 this paper we adopt this point of view.   The correspondence between  these resonance states  and those defined
 from a  meromorphic  continuation of $F(z), \Im z >0$ was discussed in \cite{abbfr}.

Dynamical resonance behavior of periodically perturbed Hamiltonians has already been obtained
for example, in  \cite{cos} and \cite{sowe2}, in a formulation similar to ours.

Here, we first
obtain the Fermi Golden Rule for a generic set of perturbations. Also, we prove directly the Mourre
estimate for the corresponding Floquet Hamiltonian. For this reason, our results
hold away from thresholds.

Previous works need a local decay, pointwise in time, which could hold  at thresholds as well,  (see e.g  \cite{JeNe}  for discussion).

The article is organized as follows. First, we give a brief review of the results of \cite{CGHu}  in Section 2
  and of the Howland-Floquet formalism in  Section
3.  In Section 4 we describe the model
studied in the paper. The resonance states for the associated  Floquet operators are
described in Section 5 and  in Section
6 we show that  the Fermi golden rule holds for a generic class of  time dependent perturbations. Finally in Section 7  we derive a local decay in time on the propagator associated to the  the time dependent Schr\"odinger equation.

\section{Mourre estimates and resonances}

Let  $H$ be a self--adjoint operator acting on a Hilbert space $
\mathcal{H}$. For every Borel set $\Omega$, denote by $E_\Omega (H)$ the  spectral projector of the selfadjoint operator $H$ associated  with $\Omega$. We will say that $H$ satisfies a Mourre estimate \cite{Mo}  on an interval $ I =
(a,b) \subset \RR$ if there exists a self adjoint operator $A$
such that,
\begin{equation} \label{EM}
E_I(H)i[H,A] E_I(H) > c E_I (H)
+K,
\end{equation}
where $c>0$ and $K$ is a compact operator.

The commutator  $i[H,A] = i(HA-AH)$ may be difficult to define, due to domain problems. Its definition
requires the condition $e^{itA} D(H) \subset D(H)$  forall  $t \in \mathbb R$. Then  the estimate,
$$| i\langle H u, Av \rangle - i \langle Au, v\rangle | <
C \|u\| \ \| (H+i)v \|,$$
$\forall u,v \in D(H) \cap  D(A)$ allows  to define  $i[H,A]$ in the quadratic form sense.
Next,  we  consider multiple order commutators,
$$ ad^{(1)}_A (H)\equiv i[H,A]
$$
and for $n\in \mathbb N$
$$
ad^{(n+1)}_A (H)\equiv i[ad^{(n)}_A (H),A].
$$
Suppose that  $ad^{(j)}_A (H), \;\;j=1 \cdots  \nu$  are defined as  H-bounded operators, and    \eqref{EM} holds with $K=0$   (which implies that $H$ has no eigenvalue in $I$). Then for  some $s >1/2$,  the weighted resolvent
$$(A-i)^{-s}(H-z)^{-1}(A+i)^{-s};  \Re z \in I, \Im z >0
$$ has a limit in the  bounded operator sense on $\mathcal{H}$ as $\Im z $ approaches $0$.  Moreover, for all $\varphi \in \mathcal{H}$
the function,
$$
\lambda \in I \to g_\varphi (\lambda)\equiv \langle (A-i)^{-s}(H-\lambda-i0)^{-1}(A+i)^{-s} \varphi, \varphi\rangle
$$
admits derivatives up to order $n-1$ on $I$.

 Further, suppose that $H$ is a self--adjoint operator having a
simple eigenvalue $E_0 \in I $, {\it embedded in the continuous spectrum}.  Let  $\varphi_0$ be the associated eigenvector,
$H\varphi_0 =E_0\varphi_0$, $ \Vert \varphi_0 \Vert =1$. Denote by $P$ the corresponding eigenprojector and $ Q=\mathbb I -P$.
Consider the perturbed Hamiltonian,
$$
H_{\alpha}=H+\alpha W
$$
Assume also  that the operators  $ad^{(j)}_A (W), \;\;j=1 \cdots \nu,$  are  H-bounded operators. Then  for $\alpha$ small enough, the function
$$
F(z, \alpha)\equiv \langle Q W\varphi_0, (QH_\alpha Q-z)^{-1
}QW \varphi_0 \rangle, \Re z \in I $$
  has  a boundary value as $ \Im z \to 0$. Moreover $E\in I \to F(E+i0, \alpha)$
admits derivatives up to order $n-1$.

The main result in \cite{CGHu},  states the following. Let $N \geq 1$ and $ \nu > N+5$ be some integers. Under above conditions, there exists a function $g \in C^{\infty}_0 (\mathbb R)$, such that $g(\lambda)=1$ in a small interval around $E_0$ with
$ \sup \vert g\vert \leq1$,  and  complex numbers $ E_\alpha $ such that for $ \alpha$ small enough
\begin{equation} \label{main}
\langle \varphi_0, e^{-itH_{\alpha}} g(H_{\alpha})\varphi_0 \rangle = a(\alpha) e^{-iE_{\alpha  }t} + b({\alpha},t),
\end{equation}
where $a(\alpha)= 1- O(\alpha ^2)$ and $b({\alpha}, t) = O\bigg( \alpha^2 \vert  \log \vert \alpha \vert \vert (t+1)^{1-N}\bigg)$.

In the following we write $\langle \varphi_0, e^{-iH_{\alpha}t} g(H_{\alpha})\varphi_0 \rangle \approx e^{-iE_{\alpha }t}$.

Moreover for $ \alpha \in \mathbb R$, and small enough
\begin{equation} \label{ER}
E_{\alpha}=E_0 +\alpha\langle \varphi_0, W \varphi_0
 \rangle
 - \alpha ^{2} F(E_0+i0,0)+o(\alpha^2).
\end{equation}
Note  that for $\epsilon =\Im z >0$ and $\alpha\neq 0$, we have that,
$$
\Im F(z,\alpha)=\alpha ^2 \epsilon \| Q(H_\alpha-z)^{-1}QW\varphi_0\|^2 \geq 0.
$$
Hence, suppose that,
$$
\frac{\Gamma }{2}:= \Im \; F(E_0 +i0, 0) >0.$$
This necessarily gives that $E_0$ must be
embedded in the continuous spectrum of the operator $H$. Then $\langle \varphi_0, e^{-iH_{\alpha}} g(H_{\alpha})\varphi_0 \rangle $ exhibits a local exponential decay in time i.e. $\varphi_0$ is a metastable state associated to the hamiltonian $H_\alpha$.

Although  the definition  of resonances requires the strict positivity of
 $\Gamma$,  we call the energy $E_\alpha$  in \eqref{ER} a resonance for $H_{\alpha}$

%%%%%%%%%%%%%%%%%%%%%%%%%%%%%%%%%%%%%%%%%%%%%%%%%%%%%%%%%%%%%%%%%%%%%%%%%%%%%%%%%%%%%%%%%%%%%%%%%%%%%%%%%%%%%%%%%%%%%%%%%

\section{Howland formalism}\label{sec1}

In this section, we review basics facts of the  time dependent theory initiated in \cite{howland1}. Let $\{ H(t), t \in  \mathbb  R\}$ be a family of selfadjoint operators  in an Hilbert space $\HH$. Suppose that  for $t  \in \mathbb R $, $H(t)$  has a  constant domain ${D} $. Furthermore we assume that the family $\{ H(t), t \in  \mathbb  R\}$  is $T$- periodic, $T>0$,  i.e.   $H(t+ T)=H(t)$. We notice that an important part of
this theory also applies in the non periodic case.

Consider  the
abstract time--dependent Schr\"{o}dinger equation,
\begin{equation}\label{equation}
i \frac{\partial \phi}{\partial t}(t) = H(t) \phi(t); \quad \phi(0) = \phi \in \HH.
\end{equation}
 Then under adequate conditions \eqref{equation} generates a unique propagator $\{ U(t,s); ( t,s) \in {\mathbb R}^2 \}$.

Further let $\mathcal{K} = L^2(\mathbb T ;{ \mathcal{H}}), {\mathbb T}:= {\mathbb R} /
T {\mathbb Z}$ be  the complex
Hilbert space of weakly measurable, $\mathcal{H}-$valued functions
with inner product,
 $$ \langle f,g \rangle  = \int_{0}^{T}
\langle f(t),g(t) \rangle_0 \, dt \, , $$ where $ \langle .\, , \,
. \rangle_0 $ is the corresponding  inner product in
$\mathcal{H}$. Note that the enlarged space $\mathcal{K}= L^2(\mathbb T) \otimes \mathcal{H}.$

The propagator $U(\cdot, \cdot)$
induces a strongly continuous one parameter unitary group $\{ W(\sigma);
 \sigma \in \mathbb R\}$ on the space $\KK$,  defined  as
\begin{equation} \label{rel}
 W(\sigma ) \phi(t, \cdot) = U(t,t-
\sigma)\phi(t - \sigma, \cdot );  \; \forall \phi \in \HH.
\end{equation}
Moreover, the  Floquet Hamiltonian,
$$
K=-i\frac{d}{dt}\,\otimes I_x \,+H(t).
$$
 with domain ${D}(K) = \{ \phi \in \HH; \; K\phi \in \HH \}$ is precisely the infinitesimal generator of $W(\sigma)$, that is,
  $$ W(\sigma )  = \large{e}^{-iK \sigma},  \, \sigma \in \mathbb R .$$
The idea behind this construction is that  the time-dependent evolution
in $\HH$ has been turned into a time-independent
problem in the Floquet space $\KK$.

%%%%%%%%%%%%%%%%%%%%%%%%%%%%%%%%%%%%%%%%%%%%%%%%%%%%%%%%%%
\section{Time dependent Hamiltonian}
%%%%%%%%%%%%%%%%%%%%%%%%%%%%%%%%%%%%%%%%%%%%%%%%%%%%%%%%%%

We now use the Floquet structure, combined with the results in \cite{CGHu},
to study resonances for a time periodic
family $H(t)$ of quantum
Hamiltonians  acting in the Hilbert space $\HH=L^2( \mathbb R^d)$.

Here,
the resonant behavior will be characterized by the local decay in time of the survival
probability,
$$P_s(t)\equiv |\langle \varphi, U(t,s)\varphi \rangle|^2 ,
$$
 for an adequate state $\varphi \in \HH$ and where $U(t,s)$ is the corresponding propagator.

However, unless  the Hamiltonian is time
independent, in which case $U(t,s)=U(t-s,0)$, the asymptotic behavior of the
survival probability will depend
on the initial time $s$. Actually, we shall obtain a result on the average value of
this quantity on a time interval of length $T$.

Here  $U_{\alpha}(t,s)$ denotes  the propagator associated to a Hamiltonian
$H_{\alpha}(t), t \in  \mathbb T$.  We now define precisely this  family  of operators.

Fix an integer $N$ and let $\nu:=  N+ 6 $. The Hamiltonian $H_{\alpha}(t)$ will be a time dependent perturbation of a free operator $ H$ acting in $\HH$ and defined as
\begin{equation}\label{free}
H=-\Delta +V,
\end{equation}
where $V: \mathbb{R}^d \to \mathbb{R}$ is a smooth function satisfying the following assumptions.
Let $ <x>:= (1 + \vert x \vert^2)^{1/2}$.
\medskip

{\bf hV1}:  { \it $V \in C^{\nu} (\mathbb{R}^d) $and it satisfies : there exists $p >2 $ such that  for all $ \alpha, \vert \alpha \vert \leq \nu$:
\begin{equation}  \label{hv}
sup_{x \in\mathbb R ^d }<x>^{p+  \alpha} \vert \partial^\alpha  V(x)  \vert  < \infty.
\end{equation}
}

Also  from \cite{rs2}, the operator  $H$ with domain $D(H) = \HH^2(\mathbb{R}^d)$ is a self--adjoint on $\HH$; here $\HH^2(\mathbb{R}^d)$ is the standard Sobolev spaces. Moreover,   the
spectrum  $ \sigma(H)= \sigma_d(H) \cup [0, +\infty)$.

The set $ \sigma_d(H) $ consists of a discrete set
of negative eigenvalues, they can accumulate at the threshold $0$. On the other hand $ \sigma_{ac}(H)= [0, +\infty) $, and because of our assumption there is neither  singular continuous spectrum or  positive eigenvalue embedded in $[0, +\infty) $.

In this work, we use standard notation to denote  different types of spectrum of a selfadjoint operator (see e.g. \cite{rs2}).

The operator $H$ does not depend on $t$, but we can visualize it in the  formalism described in the previous section. In
this sense, the corresponding free Floquet Hamiltonian is
$$
K =i\frac{\partial}{\partial t}\otimes I_x+ I_t \otimes H,
$$
acting on the extended Hilbert space $$\KK= L^2(\mathbb T;L^2(\mathbb{R}^d))=L^2(\mathbb T)\otimes
L^2(\mathbb{R}^d),$$
Here, $ \mathbb T= \mathbb R/ T\mathbb Z$, $I_t$ and $I_x$ denote  the identity operator on the spaces $L^2(\mathbb T)$ and
$L^2(\mathbb{R}^d)$ respectively. It is easy to see that $D(K) = \HH^1(\mathbb T) \otimes \HH^2(\mathbb{R}^d)$.

Further, the operator $i\frac{\partial}{\partial t}$ in $L^2( \mathbb T)$ has a discrete spectrum, with eigenvalues $n \omega \in \mathbb{Z}$, $  \omega := 2\pi/T $ and
eigenvectors
$e_n(t)={1\over \sqrt{ T}}e^{in\omega t}$. We denote the one dimensional projection
  $p_n=|e_n\rangle \langle e_n|$. Hence, the spectrum
of $K$ is
$$\sigma(K)=\sigma_{ac}(K)=\bigcup_{n  \in {\mathbb{Z}}}[n\omega ,\infty)=\mathbb{R}.
$$
On the other hand, the  pure point spectrum  of $K$ consists of the translation of
the eigenvalues of $H$  by any  $n \omega, n \in \mathbb N$.

We  also suppose that,

{\bf hV2}:  { \it the operator $H$ has a simple   eigenvalue $E_0$   with eigenvector $\varphi_0$
such that  for all $ n \in \mathbb N^*$, $\mu_n := E_0 + n\omega \notin \sigma_{d}(H)\cup \{0\}$.}

The assumption {\bf hV2} means that first $E_0$ is also a simple eigenvalue of Floquet Hamiltonian $K$ but it is embedded
in its absolutely continuous spectrum.  Actually, this is true for all
eigenvalues of $K$. Moreover $E_0$ is not a spectral threshold of    $K$.
In fact   without loss of generality we will  suppose  here :

{\bf hV'2}: { \it  the operator $H$ has a simple   eigenvalue $E_0$   with eigenvector $\varphi_0$
such that  $\vert E_0\vert < \omega $.}

\noindent Clearly  this implies assumption {\bf hV2}.

We now  introduce the  time dependent perturbation. Let
$$  (t,x) \in \mathbb T \times \mathbb R^d \to W(t,x) \in \mathbb R$$
be a time periodic potential, $W(x,t+T)=W(x,t); t \in \mathbb R$, $x \in \mathbb R^d$   satisfying,

{\bf hW}  { \it $W \in C( \mathbb T; C^\nu(\mathbb R ^d))$ and  there exists  $p >2 $ such that for all
$ \alpha, \vert \alpha \vert \leq 2$
\begin{equation}  \label{hw}
   \sup_{t \in \mathbb T } \sup_{x \in \mathbb R ^d }\{  <x>^{ p+ \alpha}   \vert  \partial_x^\alpha W(x,t)  \vert \} < \infty.
\end{equation}}
Then the perturbed Hamiltonian,
\begin{equation}\label{perturbed}
H_{\alpha}(t) =H+\alpha W(x,t); \alpha \in \mathbb R, t \in \mathbb T
\end{equation}
is  a self adjoint  operator with a  time independent domain, $ \mathcal {D}(H_\alpha(t))$= $ {\mathcal H}^2(\mathbb R^d)$.

The corresponding selfadjoint  Floquet Hamiltonian is
\begin{equation}\label{perturbedfloquet}
K_{\alpha}=K+\alpha W(x,t),
\end{equation}
acting on the enlarged space $\KK$ with domain $ \mathcal {D}(K_\alpha)= \mathcal {D}(K)$, for all
$ \alpha \in \mathbb R$.

\subsection{Mourre estimate for the Floquet operator}

%%%%%%%%%%%%%%%%%%%%%%%%%%%%%%%
Consider the following operator  $ D:= -i {\nabla}(-\Delta + 1)^{-1}$, it is a bounded operator on $L^2(\mathbb R^d)$. Also, set
$$
A={1\over 2}(x\cdot D+ D \cdot x), $$
Then $A$ is an essentially selfadjoint operator on $L^2(\mathbb R^d)$  such that
$ e^{itA }{\mathcal H}^2(\mathbb R^d) \subset {\mathcal H}^2(\mathbb R^d)$ (see e.g. \cite{yo}). We denote by $$
B=I_t \otimes A,
$$
 the corresponding conjugate operator acting on the  space $\KK$. It is easy to see that in the form sense on $C_0^\infty(\mathbb R ^d) \times C_0^\infty(\mathbb R ^d) $
\begin{equation} \label{comml}
i[-\Delta,A]=-2 \Delta (-\Delta + 1)^{-1},
\end{equation}
and then it extends  to a bounded  selfadjoint operator  in $L^2(\mathbb R^d)$.

We  also have the following

\begin {lemma} \label{comm}   The commutator $i[H,A]$, defined  in the form sense  on $C_0^\infty(\mathbb R ^d) \times C_0^\infty(\mathbb R ^d) $,  extends to a  bounded  selfadjoint
operator in $L^2(\mathbb R^d)$.
Moreover, the multiple commutators $ad_{A}^j (H)$ are bounded, for $j= 1...  \nu$.
\end {lemma}

\proof  Formally we have that,
$$i[V,A]=  i[xV,D]  +i   [D,x]V.$$
Clearly, the commutator,
$$  [xV,D] =  i(xV)' (-\Delta + 1)^{-1} +  i\nabla( -\Delta + 1)^{-1} (2 \nabla.(xV)' - (xV)'')( -\Delta + 1)^{-1},$$
extends to a compact operator in $L^2(\mathbb R^d)$.

Indeed,  on the one hand our assumption hV1 implies that  the operators $ (xV)' (-\Delta + 1)^{-1} , (xV)'' (-\Delta + 1)^{-1}  $ are  compact; hence,  $[D,xV] $ as a sum of two compact operators, is also compact.

Now  by using
$$ [x,D] = i(-\Delta + 1)^{-1}  -2i \Delta ( -\Delta + 1)^{-2}$$
 and the assumption hV1, it is easy to see that $ [x,D] V$ as well $V[x,D] $ are compact operators.

Therefore these arguments together with the identity \eqref{comml} prove the first part of the lemma.

Now computing  $ad_{A}^j (-\Delta); j= 1...  \nu$. We get
$$ ad_{A}^j (-\Delta) = (-\Delta +1)^{-j} q\bigg( -2 \Delta((-\Delta +1)^{-1}\bigg)$$
where $q$ is a polynomial of degree $j$, so it is bounded operator on $\HH$.

Finally, the multiple commutators  $ad_{A}^j (V); j =1 ... \nu$ involve  some combination of higher order derivatives of $V$ and  bounded operators. Thus, these are also bounded up to the order $\nu$,
 by our assumption hV1.  \qed

Since we have
$$
i[K,B]=I_t \otimes i[H,A],
$$
then, the Lemma \ref{comm} implies that  the commutator $i[K,B]$
is a bounded   selfadjoint operator in $\KK$. Moreover,   the higher order commutators exist as  bounded operators.
Similarly, by using  assumption {\bf hW}, we can prove that $ad_{B}^j (W)$, $j= 1...  \nu$ also are bounded operators.

We now construct a Mourre estimate for the free Floquet Hamiltonian $K$.
Let  $J_{n_0}= (e_{n_0}+J_-,  e_{n_0}+ J_+), J_- < 0, J_+ >0$ be a   small interval around the energy $e_{n_0}= E_0 + n_0\omega$   so that $J_{n_0}$ contains no other eigenvalues of $H$.

Next let
$$E_{J_{n_0}} (K)= \oplus _{n \in \mathbb Z} p_n  \otimes E_{J_{n_0}} ( H+n\omega),$$
be the spectral projector of $K$ associated with the interval $J_{n_0}$. Note also  that $E_{J_{n_0}} (K) = E_{J_0} (K-n_0\omega)$.
\begin{lemma} \label{Mourre} Suppose {\bf hV1} and {\bf hV'2}. Let $n_0 \in \mathbb Z$ and
$J_{n_0}$ be  the energy interval defined above. If $\vert J_{n_0}\vert$ is small enough, then the operator $K$ satisfies a local Mourre estimate. Explicitly, there exists a constant $c>0$ independent of $n_0$ and a compact operator $L$ such that
\begin{equation}\label{mourre}
E_{J_{n_0}} (K) i[K,B]E_{J_{n_0}} (K) \ge cE_{J_{n_0}} (K)+L.
\end{equation}
\end{lemma}

\begin{proof} We can choose  $n_0=0$,  the lemma follows for any $n_0 \in \mathbb Z$ by the same arguments.

From \eqref{comml}, we need to bound from  below  modulo a compact operator, the  following operator
$$   E_{J_0}(K) I_t \otimes -\Delta(-\Delta +1)^{-1}E_{J_0} (K)= $$
$$ E_{J_0} (K) I_t \otimes H(-\Delta +1)^{-1}E_{J_0} (K)  -   E_{J_0} (K) I_t \otimes V(-\Delta +1)^{-1}E_{J_0} (K) = $$
$$ E_{J_0} (K) I_t \otimes H(H +1)^{-1}E_{J_0} (K)  +  E_{J_0} (K) I_t \otimes H(H +1)^{-1}V(-\Delta +1)^{-1}E_{J_0} (K)-$$
$$ E_{J_0} (K) I_t \otimes V(-\Delta +1)^{-1}E_{J_0} (K).$$
We first consider the term
$$ E_{J_0} (K) I_t \otimes H(H +1)^{-1}E_{J_0} (K) = \sum_{n \in \mathbb Z} p_n \otimes H (H +1)^{-1}E_{J_0} ( H+n).$$
We  note that if $\vert J_0 \vert$ is small enough then $E_{J_0} ( H+n)=0$, for $n \geq 1$.

In the other hand, if $n <0$ then,
$$H  (H +1)^{-1}E_{J_0} ( H+n) \geq (J_-  - n )(1 + J_- -n)^{-1}E_{J_0} ( H+n)  \geq c E_{J_0} ( H+n)$$
where $c= (1+ J_- )(2 + J_- )^{-1}>0$.

Hence,
$$ E_{J_0} (K) I_t \otimes H(H +1)^{-1}E_{J_0} (K)  \geq  c \sum_{n <0} p_n \otimes  E_{J_0} ( H+n) +  L_1= $$
 $$
  c \sum_{n \in \mathbb Z } p_n \otimes  E_{J_0} ( H+n) +  L_1= c E_{J_0} (K)+ L_1 $$
 where,  $$L_1:= (H(H +1)^{-1}-c)p_0\otimes E_{J_0}(H) = (E_0(E_0+1)^{-1}-c)p_0\otimes E_{J_0}(H), $$

 For $\vert J_0 \vert$  small enough.    $L_1$ is  a rank one operator   as the product of two rank one operators, $p_0$ and  $E_{J_0}(H)$.

Now, let
$$ L_2 :=  E_{J_0} (K) I_t \otimes H(H +1)^{-1}V(-\Delta +1)^{-1}E_{J_0} (K).$$

Because of our assumptions,
 the operator
$H(H +1)^{-1}V(-\Delta +1)^{-1}$ is compact.

On the other hand, for negative $n$,
$$\Vert H(H +1)^{-1}V(-\Delta +1)^{-1} E_{J_0} ( H+n) \Vert = o( 1/ \vert n \vert).$$

Indeed, since $\Vert E_{J_0} ( H+n) (H +1)^{-1}\Vert = o( 1/ \vert n \vert)$, by  using the resolvent equation, $\Vert E_{J_0} ( H+n) (-\Delta  +1)^{-1}\Vert $ and then $\Vert H(H +1)^{-1}V(-\Delta +1)^{-1} E_{J_0} ( H+n) \Vert = o( 1/ \vert n \vert)$.
Finally
$$ L_2 = \sum_{n<0}p_n \otimes E_{J_0} ( H+n) H(H +1)^{-1}V(-\Delta +1)^{-1} E_{J_0} ( H+n) $$
is compact, as a uniform  norm limit of compact operators. Clearly these arguments lead to the compacity of $L_3 =E_{J_0} (K) I_t \otimes V(-\Delta +1)^{-1}E_{J_0} (K)$. Hence,
we conclude that in the quadratic form sense in  $\KK$, there exist a positive constant $c$ and three compact operators, $L_1,L_2, L_3$ such that,
\begin{equation}
 E_{J_0} (K) I_t \otimes -\Delta(-\Delta +1)^{-1}E_{J_0} (K) \geq c  E_{J_0} (K)+ L_1 + L_2 + L_3.
 \end{equation}

Now from  Lemma \ref{comm}  we can repeat the same lines of arguments as above to the operator $L_4:= E_{J_0} (K)I_t \otimes [V,A] E_{J_0} (K)$ and then $L_4$ is again compact. This finishes the proof of the lemma.
\end{proof}
%%%%%%%%%%%%%%%%%%%%%%%%%%%%%%%%%%%%%%%%%%%%%%%%%%%%%%%%%%%%%%%%%%%%%%%%%%%%%%%%%
%%%%%%%%%%%%
\section{resonances ladder por the Floquet operator}

In this section we want to derive the localization of resonances   for the  Floquet Hamiltonian
\eqref {perturbedfloquet} associated to the  eigenvalue $E_0$ of $H$.

To this end, we introduce further notations concerning the spectrum of the free
Floquet operator. Let  $\varphi_0 $ be the eigenvector of $H$ associated with $E_0$,  $
H\varphi_0=E_0 \varphi _0$. We denote the  orthogonal eigenprojector $\pi_0 := |\varphi_0 \rangle \langle \varphi_0 |$ on $\HH$, onto the one dimensional subspace generated by $\varphi_0$.

Then $ \{E_0 + n \omega; n \in \mathbb Z\} \subset \sigma_{pp}(K)$ and from {\bf hV'2},  if $E$ is any other eigenvalue of $H$ then $\{E+ n \omega; n \in \mathbb Z\}   \cap \{E_0 + n \omega; n \in \mathbb Z\} = \emptyset$.

Next let $ n_0 \in \mathbb Z$ and  $J_{n_0}$  be the interval around $E_0 + n_0\omega$ defined in the previous section. Denote by $f_{n_0}=e_{n_0}(t) \otimes \varphi_0(x)$  the  eigenvector of the operator $K$, associated to the eigenvalue $E_0 + n_0\omega$.

We also use  $ P_{n_0}=|f_{n_0}\rangle \langle f_{n_0}| = p_{n_0} \otimes \pi_0$ and $Q_{n_0}=I_{\KK}-P_{n_0}$.

 Following Section 2,   we need to consider the function,
\begin{equation}
F(z,\alpha)=  \langle f_{n_0},WQ_{n_0}(K_\alpha-z)^{-1}Q_{n_0}W f_{n_0}\rangle,
\end{equation}
where $\Re z \in J_{n_0} $, $ \Im z = \epsilon >0$. Set
\begin{equation} \label{Wn}
W_n(x)={1\over \sqrt{T}}\int_0^T e^{-int}W(x,t)dt.
\end{equation}
 We have the following
%%%%%%
\begin{lemma} \label{Im} i) $F(E_0+n_0 \omega+ i0,0)$ exists and it is independent of  $n_0$. For any $\epsilon >0$,
this quantity is given by
\begin{multline}
 F(E_0+i\epsilon, 0) =\sum_{n\not= 0} \langle W_{n} \varphi _0,
 (H+ n \omega -E_0 - i \epsilon )^{-1}W_{n} \varphi _0\rangle + \\
 \langle W_{0} \varphi _0,
(1-\pi _0)(H-E_0 -i \epsilon )^{-1}W_{0} \varphi _0\rangle. \label{F}
\end{multline}
ii) We also have
\begin{multline} \label{Gamma}
 \frac{ \Gamma}{2} = \lim_{\epsilon \to 0}\Im F(E_0 +i \epsilon,0) =\\
 \sum_{n < 0} \Im \langle W_{n} \varphi _0,
 (H+ n \omega -E_0 - i 0)^{-1}W_{n} \varphi _0\rangle.
\end{multline}
\end{lemma}

\proof Let $ z = E_0 +  {n_0}\omega + i \epsilon; \epsilon >0$. The existence of the boundary value is a consequence of Section 2 and our assumptions.  Clearly,
\begin{equation}\label{decK}
(K-z)^{-1}=\sum _{n\in{\mathcal Z}}p_n \otimes (H+ n \omega -z)^{-1}
\end{equation}
Since  $P_{n_0}=p_{n_0} \otimes \pi_0$ and  $Q_{n_0}=I_t\otimes I_x-P_{n_0}$, we have that,
$$
Q_{n_0}(K-z)^{-1}=
\sum _{n\in{\mathcal Z}}(p_n\otimes (H+n \omega -z)^{-1}-p_{n_0}p_n \otimes
\pi _0(H+n \omega -z)^{-1})
$$
For $n={n_0}$, the quantity in the sum is just, $
p_{n_0} \otimes (1-\pi _0)(H + n_0 \omega-z)^{-1},$
while for $n\neq {n_0}$, this term is, $
p_n \otimes (H+n\omega -z)^{-1}$.

Then, we obtain,
\begin{eqnarray*}
F(z, 0)=\sum_{n \not = {n_0}} \langle e_{n_0}\otimes \varphi _0,
Wp_n\otimes (H+ (n-n_0) \omega -E_0 - i \epsilon )^{-1}We_{n_0} \otimes \varphi _0\rangle + \\
\langle e_{n_0}\otimes \varphi _0,
Wp_{n_0} \otimes (1-\pi _0)(H  -E_0 -i \epsilon )^{-1}We_{n_0} \otimes \varphi _0\rangle,
\end{eqnarray*}
and then
\begin{eqnarray*}   F(z, 0)=\sum_{n\not= {n_0}} \langle W_{n-n_0} \varphi _0,
 (H+ (n-n_0) \omega -E_0 - i \epsilon )^{-1}W_{n-n_0} \varphi _0\rangle + \\
\langle W_{0} \varphi _0,
(1-\pi _0)(H  -E_0-i \epsilon )^{-1}W_{0} \varphi _0\rangle.
\end{eqnarray*}
This proves \eqref{F}.

By assumption {\bf hV2}, for $n>0$, $ E_0 -n \omega \in \rho(H)$, where $\rho(H)$ denotes the resolvent set of $H$.

Therefore,
$$\lim_{\epsilon \to 0} \Im \langle W_{n} \varphi _0,
 (H+ n \omega -E_0 - i \epsilon )^{-1}W_{n} \varphi _0\rangle =0.$$
Also, because $E_0 \in \rho((1-\pi _0)H)$ we have
$$
\lim_{\epsilon \to 0} \langle W_{0} \varphi _0,
(1-\pi _0)(H  -E_0-i \epsilon )^{-1}W_{0} \varphi _0\rangle =0.
$$
On the other hand for $n<0$ and for all positive numbers $ \epsilon $
$$ \Im \langle W_{n} \varphi _0,
 (H+ n \omega -E_0 - i \epsilon )^{-1}W_{n} \varphi _0\rangle \geq 0.$$
 Hence these last three estimates prove \eqref{Gamma}.
\qed
{\remark  Since $E_0 -n \omega> 0 = \inf \sigma_c(H)$ if $ n <0$ we can  express the strict  positivity of $\Gamma$ in terms of  strict positivity of the derivative of   the spectral measure   $E_{(-\infty,\lambda]}(H) = E_{(E_0,\lambda]}(H) $.  Indeed we know that in the distributional sense
$$
\Im \langle W_n\varphi _0,  (H+ n\omega -E_0-i0)^{-1}W_n \varphi _0\rangle
= \frac{d\langle W_n\varphi _0, E_{(-\infty,\lambda]}(H) W_ n \varphi_0\rangle}{d \lambda}|_{\lambda = E_0-n \omega}.
$$
This last quantity is strictly positive if $W_n \varphi_0$ has its   spectral support around  the energy $E_0-n \omega$.    We show below that the condition
$$E_{(E_0,\lambda]}(H)W_n \varphi \neq 0$$
is satisfied by a generic class of potentials $W$. }
\medskip

Then we conclude that
\begin{theorem} \label{resF}
Under conditions {\bf hV1}, {\bf hV'2}, ({\bf hV2})  for $\alpha$ small enough, the Floquet Hamiltonian admits resonances of the form
\begin{equation}\label{resonances}
E_{n,\alpha }  = E_0 + n\omega + \alpha c_1 -  \alpha^2 c_2  + o_n(\alpha^2); n \in \mathbb Z.
\nonumber \end{equation}
where $c_1 =  \frac{1}{T} \int_{[0,T) \times \mathbb R^d}  \vert  \varphi_0(x)\vert^2W(x,t)dtdx$ and $c_2= F(E_0+ i0,0)$. In particular the width of these resonances is
\begin{equation} \label{width}
\Gamma  =  2 \alpha^2  \Im F(E_0+ i0,0).
\end{equation}
where $\Im F(E_0+ i0,0) $ is given by \eqref{Gamma}.

Moreover, we have that for each $n \in \mathbb Z$,

\begin{equation} \label{taux decay}
\langle e_{n}\otimes \varphi_0, g( K_\alpha)e^{-i sK_\alpha} e_n\otimes \varphi_0\rangle \approx e^{-iE_{n,\alpha}s }.
\end{equation}
in the sense of  \ref{main}.
\end{theorem}

\section{The Fermi golden rule.}

In this section we want to show that   the width $\Gamma$  defined by \eqref{Gamma} is  strictly positive for a generic class of perturbations $W$. To this end we use an   eigenfunction expansion for the operator $H$.

From  \cite{k}, we know  that the operator  $H$ has  a complete set of
real generalized eigenfunctions, $ \{\varphi(k, .); k \in
\mathbb R \}$   which are bounded and uniformly continuous.

By using standard arguments of the eigenfunction expansion theory (see e.g. \cite{rs2}) we have the following Lemma

\begin{lemma} Suppose  {\bf hV} and {\bf hW}. Then
\begin{multline}
\frac{\Gamma}{2}= \sum_{n  >0}  \frac {1}{2\sqrt{\mu_n}}\bigg( \vert \int _{\mathbb R^d}dx W_n (x)\varphi_0(x) { \varphi(\sqrt{\mu_n},x)}\vert ^2 + \\ \vert \int_{\mathbb R^d} dx W_n (x)\varphi_0(x)  { \varphi(- \sqrt{\mu_n},x)}\vert ^2  \bigg).
\end{multline}
where $W_n $ are defined in \eqref{Wn}.
\end{lemma}

Now  introduce  the normed space $\bf W$ as  the set of  real perturbations $W \in C( \mathbb T ; C^\nu(\mathbb R ^d))$ satisfying,
\begin{equation} \label{newnorm}
\Vert W \Vert_{\bf W} := \sup_{x \in \mathbb{R}^d } \bigg( \frac {1}{T}\sum_{\alpha, \vert \alpha \vert < \nu}  <x>^{2 + \alpha}  \big(\int_0^T \vert  \partial_x^\alpha W(x,t)  \vert^2 dt) \big)^{1/2} \bigg)< \infty .
\end{equation}
For each $n >0$, introduce the sets:
$$  D_{\pm, n}:= \{ W \in  {\bf W} \quad  { \rm s.t. }  \int dx W_n(x) \varphi_0(x)  { \varphi( \pm \sqrt{\mu_n},x)}
\not= 0\} $$
\begin{lemma}  For each $n \in  \mathbb N$,  $D_{\pm, n}$ is a   dense open  subset  of  $ {\bf W} $.
\end{lemma}

\begin{proof} Consider the  linear application $I :  {\bf W} \to \mathbb{R}$ defined as,
$$ I(W) :=  1/T \int dx W_n(x) \varphi_0(x)  { \varphi(\sqrt{\mu_n},x)};\;  W \in  {\bf W}.$$
Then  $I$ is continuous map, since we have,
$$\vert I(W) \vert ^2 \leq  C \sup_{x \in \mathbb R ^d }(  <x>^{2}  \int_0^T \vert  W(x,t)  \vert^2 dt)\leq C \Vert W \Vert^2_{\bf W}$$
where  $C := \int dx \varphi_0(x). \vert  \frac  { \varphi(\sqrt{\mu_n},x)}{<x>^{2}}\vert$. Note that $C$ can be bounded  independently of $n$. Then,  $  D_{+, n}$
is an open subset of $ {\bf W} $.

Moreover, suppose that $W \notin D_{+, n} $. We know that there exists some
real point $x_0 \in \mathbb R ^d$ such that $\varphi_0(x)  { \varphi(\sqrt{\mu_n},x)} \not= 0$ and by  a simple continuity  argument the same is  true on some neighborhood  $ \nu(x_0)$ of $x_0$. Denote by $\chi$ a $C^\infty $ positive  function  with support  in $ \nu(x_0)$, then for each $l \in  \mathbb N$, $ W^{l}:= W +  \frac{1}{l+1} .e ^{- int}\chi \in D_{+, n} $ and $\Vert W - W^{l} \Vert_{\bf W} \to 0$ as $l \to \infty$. This shows  that $D_{+, n} $ is dense in ${\bf W}$. Evidently the same arguments hold for $ D_{-, n} $.
\end{proof}
Then  we obtain,

\begin{theorem}  There exists an  dense open subset  $D$  of  $ {\bf W}$, such that for any $W \in D$,  $\frac{\Gamma}{2}>0$.
  \end{theorem}
\begin{proof}
 Let  $ D:= \bigg(\cup_{ n \in  \mathbb N} D_{+, n}\bigg) \cup \bigg( \cup_{ n \in  \mathbb N} D_{-, n}\bigg)$ and  then apply the last lemma.
 \end{proof}

\section{Average Decay of the propagator}

In this last section we derive  the result on exponential decay, but in terms of the original propagator.
\begin{theorem}
Under conditions stated above,  the propagator associated to the time dependent Hamiltonian
$H(t)$ satisfies,
$$
\frac {1}{T}\int_0^T \langle \varphi_0, U_{\alpha}(t+s, t) \varphi_0\rangle dt =  \tilde a(\alpha)e^{-i \tilde E_{\alpha}s} + \tilde b(\alpha) $$
where $\tilde E_{\alpha}= E_0 + \alpha c_1 - \alpha^2 c_2  + o(\alpha^2)$,
 $\tilde a(\alpha)= 1+O(\alpha^2)$ and $\tilde b= O( \alpha^2 \vert  \log \vert \alpha \vert) $.
\end{theorem}
\begin{proof}
For any  function $a\in L^{2}(\mathbb T)$, we have that,
$$
\langle a(t)\otimes \varphi_0, e^{-is K_{\alpha}} a(t)\otimes \varphi_0\rangle =
\int_0^{T}\langle a(t) \varphi_0, U_{\alpha}(t,t-s) a(t-s)\varphi\rangle \, dt ,
$$
In particular, choose $ a(t) := e_{n_0}(t)$. Then
$$
\langle e_{n_0} \otimes \varphi_0, e^{-is K_{\alpha}} e_{n_0} \otimes \varphi_0\rangle =
 \frac{1}{T} e^{-i{n_0}\omega s }\int_0^{T}\langle  \varphi_0, U_{\alpha}(t,t-s) \varphi_0\rangle \, dt,
$$
and since $t \in \mathbb R \to  \langle  \varphi, U_{\alpha}(t,t-s) \varphi\rangle $
is periodic with period $T$,
\begin{equation} \label{f0}
\langle e_{n_0} \otimes \varphi_0, e^{-i\sigma K_{\alpha}} e_{n_0} \otimes \varphi_0  \rangle =
\int_0^{T}\langle  \varphi_0, U_{\alpha}(t+s,t) \varphi_0  \rangle \, dt .
\end{equation}

Further by the Theorem \ref{resF} and \eqref{taux decay} we have that
$$\langle e_{n_0}\otimes \varphi_0, g( K_\alpha)e^{-i sK_\alpha } e_{n_0}\otimes \varphi_0\rangle \approx e^{-is E_{n_0 , \alpha}}.$$
Put $s=0$ in this last relation, then  $\langle e_{n_0}\otimes \varphi_0, g( K_\alpha) e_{n_0}\otimes \varphi_0\rangle = 1 -  O( \alpha^2 \vert  \log \vert \alpha \vert)  $
or equivalently $\langle e_{n_0}\otimes \varphi_0, ( I_t -g( K_\alpha) )e_{n_0}\otimes \varphi_0\rangle = O( \alpha^2 \vert  \log \vert \alpha \vert)  $.
Then
$$\langle e_{n_0}\otimes \varphi_0, e^{-is K_\alpha s} e_{n_0}\otimes \varphi_0\rangle = (1 + O(\alpha^2) )e^{-iE_{n_0, \alpha}\sigma}  + O( \alpha^2 \vert  \log \vert \alpha \vert) .$$
So by \eqref{resonances} this proves the theorem.
\end{proof}

\noindent {\bf Acknowledgements.}  C. Fernandez was  partially supported  Chilean Science Foundation {\em Fondecyt} under
Grant 1100304 and Ecos-Conicyt C10E01.
P. Briet would like to thank as well  G. Raikov ({\em Fondecyt} under
Grant 1090467)  and the Faculty of
 Mathematics of Pontificia Universidad Cat\'olica de Chile for the warm hospitality extended to him.

\end{document}